\documentclass[letterpaper,11pt]{amsart}
\usepackage{amsmath,amsthm,amssymb,mathtools,setspace,dsfont,color, marginnote}
\usepackage{float,caption}
\usepackage{wrapfig}
\usepackage[backref=page]{hyperref}
\usepackage[noabbrev]{cleveref}
\usepackage[export]{adjustbox}
\renewcommand{\thispagestyle}[1]{}

\definecolor{amber}{rgb}{1.0, 0.49, 0.0}

\theoremstyle{plain}
\newtheorem{theorem}{Theorem}[section]

\newtheorem{lemma}[theorem]{Lemma}
\newtheorem{corollary}[theorem]{Corollary}
\theoremstyle{definition}
\newtheorem{definition}[theorem]{Definition}

\newtheorem{problem}[theorem]{Open Problem}
\theoremstyle{remark}
\newtheorem{remark}[theorem]{Remark}
\newtheorem{fact}[theorem]{Fact}

\newcommand{\R}{\mathbb{R}}

\DeclarePairedDelimiter\norm{\lVert}{\rVert}
\DeclarePairedDelimiter\abs{\lvert}{\rvert}
\DeclarePairedDelimiter{\ceil}{\lceil}{\rceil}
\DeclareMathOperator*{\argmin}{arg\,min} 
\DeclarePairedDelimiter\ang{\langle}{\rangle}

\usepackage[utf8]{inputenc} % allow utf-8 input
\usepackage[T1]{fontenc}    % use 8-bit T1 fonts
\usepackage{hyperref}       % hyperlinks
\usepackage{url}            % simple URL typesetting
\usepackage{booktabs}       % professional-quality tables
\usepackage{amsfonts}       % blackboard math symbols
\usepackage{nicefrac}       % compact symbols for 1/2, etc.
\usepackage{microtype}      % microtypography
\usepackage{xcolor}         % colors
\usepackage{algorithm}
\usepackage{algorithmic}

\usepackage{amsmath}
\usepackage{amssymb}
\usepackage{mathtools}
\usepackage{amsthm}

\subjclass[2010]{}

\begin{document}

\title[Approximate Real Symmetric Tensor Rank]{\mbox{}
\vspace{-1in}\\ 
Approximate Real Symmetric Tensor Rank}

\author{Alperen A. Erg\"ur, Jesus Rebollo Bueno, Petros Valettas}
\address{ }
\email{}
\thanks{}

\maketitle

\begin{abstract}
  We investigate the effect of an $\varepsilon$-room of perturbation tolerance on symmetric tensor decomposition. To be more precise, suppose a real symmetric $d$-tensor $f$, a norm $\norm{\cdot}$ on the space of symmetric $d$-tensors, and $\varepsilon >0$ are given.  What is the smallest symmetric tensor rank in the $\varepsilon$-neighborhood of $f$? In other words, what is the symmetric tensor rank of $f$ after a clever $\varepsilon$-perturbation? We prove two theorems and develop three corresponding algorithms that give constructive upper bounds for this question.  With expository goals in mind; we  present probabilistic and convex geometric ideas behind our results, reproduce some known results, and point out open problems.
\end{abstract}

\section{Introduction}
Tensors encode fundamental questions in mathematics and complexity theory, such as finding lower bounds on the matrix multiplication exponent, and have been extensively studied from this perspective \cite{derksen2016nuclear,landsberg2010ranks,pratt-waring,fulvio2022rank}.  In computational mathematics, tensor decomposition-based methods gained prominence in the 90s \cite{comon-independent}, and became a common tool for learning latent (hidden) variable models after \cite{anandkumar-tensor}. Tensor decomposition-based methods are now broadly used in application domains ranging from phylogenetics to community detection in networks. We suggest \cite{lim-survey} as an excellent survey for clarifying basic concepts and for many examples of tensor computations. Tensor decomposition-based methods are also used for a large range of tasks in machine learning such as training shallow and deep neural nets \cite{tensorshallownet,tensordeepnet}, ubiquitous applications of the moments method \cite{moitra-algorithmic, ge-survey}, computer vision applications \cite{tensorvision}, and much more: See \cite{acar-unsupervised, sidiropoulos-survey} for surveys of results available as of 2008 and 2017, respectively. 

As opposed to using arbitrary tensors without any structure, the usage of symmetric tensors appears as a common thread in wide-ranging applications of tensor-decomposition based methods. This is the main focus of our paper: the real symmetric decomposition of real symmetric tensors. Let us be more precise: 
\begin{definition} [Symmetric Tensor Rank]
Let $f$ be an $n$-variate real symmetric $d$-tensor and $S^{n-1}:= \{ u \in \mathbb{R}^n : \norm{u}_2=1 \}$. The smallest $m\in \mathbb N$ for which there exist $c_1,\ldots,c_m\in \mathbb R$ and $v_1, \ldots, v_m \in S^{n-1}$ so that
    \[
    f = \sum_{i=1}^m c_i \; \underbrace{v_i \otimes v_i \otimes \cdots \otimes v_i}_\text{d times}
    \]
is called the symmetric tensor rank and we denote this rank by ${\rm srank(f)}$.
\end{definition}
Symmetric tensor rank is sometimes called CAND in signal processing, and can be named real-Waring rank after identification with homogeneous polynomials \cite{eliassymmetric}. Analogous definitions for asymmetric tensors are called CANDECOMP, PARAFAC, or CP as a short for these two names, \cite{limsymmetric}.  We emphasize that in our definition real symmetric tensors are decomposed into rank-1 real tensors, whereas in basic references such as \cite{limsymmetric,eliassymmetric} the main focus is on the decomposition of real symmetric tensors into complex rank-one tensors. One reason for using complex decomposition is to be able to employ tools from algebraic geometry which work better on algebraically closed fields, see e.g. the delightful paper \cite{nie-lowrank}. Our aim in this paper is to use convex geometric tools to take advantage of the beauties of real geometry: being an ordered field makes the geometry over the reals (and rank notions) intrinsically different than the complex ones.
 
It is known that the tensor rank on reals is not stable under perturbation: It is typical/expected for designers of tensor decomposition algorithms to exercise caution not to let noise obscure a low-rank input tensor as a high-rank one. In a similar spirit to the smoothed analysis \cite{spielmansmoothed}, we suggest viewing the inherent existence of error in real number computations as an advantage rather than an obstacle. More formally, we propose to relax the ${\rm srank}$ notion with an $\varepsilon$-room of tolerance.
\begin{definition}[Approximate Symmetric Tensor Rank] \label{approximaterank}
Let $\norm{\cdot}$ denote a norm on the space of $n$-variate real symmetric $d$-tensors. Given a symmetric $d$-tensor $f$, we define the $\varepsilon$-approximate rank of $f$ with respect to $\norm{\cdot}$ as follows:
    \[
    {\rm srank}_{\norm{\cdot}, \varepsilon}(f) := 
    \min \{ {\rm srank}(h) : \| h-f\| \leq  \varepsilon \}.
    \]
\end{definition}
Our main results, Theorem \ref{thm:intro1}, Theorem \ref{thm:alg-greedy}, Theorem \ref{thm:intro2}, Theorem \ref{enhancedsparse}, and Corollary \ref{HSrank} show that ${\rm srank}_{\norm{\cdot}, \varepsilon}$ behave significantly different than its algebraic counterpart  ${\rm srank(f)}$. 

From an operational perspective, one might prefer to use an ``efficient'' family of norms instead of using an arbitrary norm as in definition \ref{approximaterank}. Although some of our theorems hold for arbitrary norms, our main focus is on perturbation with respect to $L_p$-norms. This is due to the existence of efficient quadrature rules to compute $L_p$-norms of symmetric tensors \cite{mauricio, townsend}. 

The rest of the paper is organized as follows: In section \ref{concepts} we introduce the vocabulary and basic concepts; in Section \ref{section:Energy}, Section \ref{section:Maurey}, and Section \ref{section:Frank-Wolfe} we present three constructive estimates, based on three different ideas, for approximate rank.  Section \ref{experiments} presents implementation details of the energy increment algorithm  (Algorithm \ref{alg:ga}). Finally, in  section 6 we consider an application to optimization.
 
\section{Mathematical Concepts} \label{concepts}
In this section, for the sake of clarity, we explicitly introduce all the mathematical notions that are used in this paper. 
\subsection{Basic Terminology and Monomial Index}
%%%%%%%%%%%%%%%%%%%%%%%%%
\label{basics}
Let $T^{d}(\mathbb{R}^n):= \mathbb{R}^{n} \otimes \mathbb{R}^{n} \otimes \cdots \otimes \mathbb{R}^n$ be the set of all $d$-tensors. Then, we consider the action of the symmetric group on the set $\{1,2,3,\ldots,d \}$, $\mathcal{S}_d$, on $T^{d}(\mathbb{R}^n)$ as follows: for $\sigma \in \mathcal{S}_d$ and $u^{(1)} \otimes u^{(2)} \otimes \ldots \otimes u^{(d)} \in T^{d}(\mathbb{R}^n)$ we have
\[ \sigma ( u^{(1)} \otimes u^{(2)} \otimes \cdots \otimes u^{(d)}) = u^{(\sigma(1))} \otimes u^{(\sigma(2))} \otimes \cdots \otimes u^{(\sigma(d))}.\]
The action of $\mathcal{S}_d$ extends linearly to the entire space $T^{d}(\mathbb{R}^n)$. A tensor $A \in T^{d}(\mathbb{R}^n)$ is called a symmetric tensor if $\sigma(A) = A$ for all $\sigma \in \mathcal{S}_d$. We denote the vector space of symmetric $d$-tensors on $\R^n$ by $P_{n,d}$. Equivalently, one can think about this space as the span of self-outer products of vectors $v \in \mathbb{R}^n$, that is,
\[ 
    P_{n,d}:= {\rm span} \{ \underbrace{v \otimes v \otimes v \otimes \cdots \otimes v}_\text{d times} \mid v \in \mathbb{R}^n \}.
\]

Now we pose the following question: Given a rank-1 symmetric tensor $v \otimes v \otimes v \in P_{n,3}$, what is the difference between $[v \otimes v \otimes v]_{1,2,1}$ and $[v \otimes v \otimes v]_{1,1,2}$? Due to symmetry, these two entries are equal. Likewise, for any element $A$ in $P_{n,d}$, two entries $a_{i_1,i_2,\ldots,i_d}$ and $a_{j_1,j_2,\ldots,j_d}$ are identical whenever $\{ i_1 , i_2, \ldots, i_d \}$ and $\{ j_1 ,j_2, \ldots, j_d \}$ are equal as supersets. This allows the use of monomial index: A superset $\{ i_1 , i_2, \ldots, i_d \}$ is identified with a monomial $x^{\alpha}:=x_1^{\alpha_1} x_2^{\alpha_2} \ldots x_n^{\alpha_n}$, where $\alpha=(\alpha_1, \ldots, \alpha_n)$, $\alpha_j$ is the number of $j$'s in $\{ i_1 , i_2, \ldots, i_d \}$, and $d=\alpha_1+\alpha_2+\ldots+\alpha_n$ is the degree of the monomial. In the monomial index, instead of listing $\binom{d}{\alpha}$ equal entries for all the supersets identified with  $x^{\alpha}$, we only list the sum of these entries once.
%%%%%%%%%%%%%%%%%%%%%%%%%
\subsection{Euclidean and Functional Norms}
%%%%%%%%%%%%%%%%%%%%%%%%%
\label{functionalnorms}
For $f \in P_{n,d}$ and $x \in S^{n-1}$, when we write $f(x)$ we mean $f$ applied to $[x,x,\ldots,x]$ as a symmetric multilinear form. For $r\in[2,\infty)$, the $L_r$ functional norms on $P_{n,d}$ are defined as 
\[ 
    \norm{f}_r := \left( \int_{S^{n-1}} |f(x)|^r \; \sigma(x) \right)^{1/r}, \quad f \in P_{n,d},
\]
where $\sigma$ is the uniform probability measure on the sphere $S^{n-1}$. The $L_{\infty}$-norm on $P_{n,d}$ is defined by
\[ 
    \norm{f}_{\infty} := \max_{v \in S^{n-1}} \abs{f(v)}. 
\]
For all $L_r$-norms, we use $B_r$ to denote the unit ball of the space $(P_{n,d}, \|\cdot\|_r)$. That is, 
\[
    B_r := \{ p \in P_{n,d} : \norm{p}_r \leq 1 \}.
\] 
We recall an important fact about $L_r$-norms of symmetric tensors established in \cite{barvinok}.

\begin{lemma}[Barvinok] \label{lem:Barv-ineq}
Let $g \in P_{n,d}$, then we have 
\begin{align*} 
\norm{g}_{2k} \leq \norm{g}_{\infty} \leq \binom{kd + n -1}{kd}^{\frac{1}{2k}} \norm{g}_{2k}.
\end{align*}
In particular, for $k \geq n \log(ed)$ we have
\begin{align*}
	\norm{g}_{2k} \leq \norm{g}_{\infty} \leq c \norm{g}_{2k}
\end{align*}
for some constant $c$.
\end{lemma}
\begin{definition}[Hilbert-Schmidt in the monomial index]
Let $p,q \in P_{n,d}$ indexed using the monomial notation, that is $p=[ b_\alpha ]_{\alpha}$ and $q = [ c_\alpha ]_{\alpha}$ where $\alpha \in \mathbb{Z}_{\geq 0}^{n}$ satisfies $\abs{\alpha}:=\alpha_1+\ldots+\alpha_n=d$. Then, the Hilbert-Schmidt inner product of $p$ and $q$ is given by
\[
\langle p,q\rangle_{\rm HS}: = \sum_{|\alpha|=d} \frac{b_\alpha c_\alpha}{\binom{d}{\alpha}}.
\]
\end{definition}
Note that in algebraic geometry literature this norm is named as Bombieri-Weyl norm.
% Clearly, $p_v(x) = \langle x,v\rangle^d = \sum_{|\alpha|=d} {d\choose \alpha }v^\alpha x^\alpha$. 
% Therefore, if  $g(x) = \sum_{|\alpha|=d} \lambda_\alpha x^\alpha$, then
% \begin{align*}
%     \langle g, p_v \rangle_W = \sum_{|\alpha| = d} \frac{{d \atopwithdelims () \alpha}{ v^\alpha \lambda_\alpha}}{{d \choose \alpha}}= \lambda_\alpha v^\alpha =g(v).
% \end{align*} 
Now, for simplicity, we define $q_v := \underbrace{v \otimes v \otimes v \otimes \cdots \otimes v}_\text{d times}$ for a $v \in S^{n-1}$, then we have the following identity: 
\begin{align} \label{nuclear}
	\max_{v\in S^{n-1}} |\langle g, q_v \rangle_{\rm HS}| = \max_{v\in S^{n-1}} |g(v)| = \|g\|_{\infty}.
\end{align} 
%%%%%%%%%%%%%%%%%%%%%%%%%
\subsection{Nuclear Norm and Veronese Body} %%%%%%%%%%%%%%%%%%%%%%%%%
\label{nuclear-norm}

We start this section by recalling the connection between norms and the geometry of the corresponding unit balls. Every centrally symmetric convex body $K \subset \mathbb{R}^n$ induces a unique norm, that is, for $x\in\R^n$
    \begin{equation}
    \label{gauge-norm}
        \norm{x}_K := \min \{\lambda >0 : x \in \lambda K \}. 
    \end{equation}
For every $v \in S^{n-1}$ we have two associated symmetric tensors: $p_v = v \otimes v \otimes v \otimes \cdots \otimes v$ and $-p_v$. Using the terminology established in \cite{orbitopes} we define the Veronese body, $V_{n,d}$, as follows:
\begin{equation} \label{Veronese}
 V_{n,d} := {\rm conv} \{ \pm p_v : v \in S^{n-1} \}. 
\end{equation}
The norm introduced by the convex body $V_{n,d}$, $\norm{.}_{V_{n,d}}$, is called the nuclear norm and it is usually denoted in the literature by  $\norm{.}_{*}$. It follows from \eqref{gauge-norm} that for every $q\in P_{n,d}$ we have
    \[ 
    \norm{q}_{*} 
    = 
    \min \left \{ \sum_{i=1}^m \abs{\lambda_i} : q = \sum_{i=1}^m \lambda_i p_{v_i} , \; v_i \in S^{n-1} \right \};
    \] 
for background material on these facts see Section 3 of the survey \cite{Go-surv}. Considering \eqref{nuclear}, one may notice that for every $q \in P_{n,d}$
    \[
    \norm{q}_{\infty}
    = 
    \max_{f \in V_{n,d}} \langle q , f \rangle_{\rm HS},
    \]
meaning that the norm introduced by $V_{n,d}$ on $P_{n,d}$ is dual to the $L_{\infty}$-norm. Then, by the duality of the norms $\norm{.}_{\infty}$ and $\norm{.}_{*}$, for every $g \in P_{n,d}$ we have
    \begin{equation} \label{nuclearSDP}
    \norm{g}_{*} 
    = 
    \max_{q \in B_{\infty}} \langle g ,q \rangle_{\rm HS}.
    \end{equation}
Formulation (\ref{nuclearSDP}) suggests a semi-definite programming approach for computing $\norm{.}_{*}$ by approximating $B_{\infty}$ with the sum of squares hierarchy. Note that this approach would yield \textit{lower bounds} for the nuclear norm that improve as the degree of sum of squares hierarchy is increased. Luckily for us, this increasing lower bounds via sum of squares idea is already made rigorous by an expert in the field and can be implemented using any semidefinite programming software \cite{nienuclear}.
\subsection{Type-2 Constant of a Norm} \label{type-2-constant}
The type-2 constant allows us to create a sparse randomly constructed approximation to a given vector with controlled error; the definition of the type-2 constant carries an essential idea to control the trade-off between error and sparsity. We will give more details and intuition on this matter in section 4. To define the type-2 constant we first need to recall that a Rademacher random variable  $\xi$ is defined by 
    \[
    \mathbb P(\xi = -1) 
    = 
    \mathbb P( \xi = 1) 
    =
    1/2.
    \]
\begin{definition} [type-2 constant] \label{type-2}
Let $\|\cdot\|$ be a norm on $\mathbb R^n$. The type-2 constant of $X=(\mathbb R^n, \|\cdot\|)$, denoted by $T_2(X)$, is the smallest possible $T>0$ such that for any $m\in\mathbb{N}$ and any collection of vectors $x_1, \ldots,x_m \in \mathbb R^n$ one has
% \marginpar{\footnotesize Comment on the Gaussian variant}	
	\begin{align}
		\mathbb E_{\xi_1,\ldots,\xi_m} \left \| \sum_{i=1}^m \xi_i x_i \right\|^2 \leq T^2 \sum_{i=1}^m \|x_i\|^2,
	\end{align} 
where $\xi_i$, $i=1,2,\ldots,m$, are independent Rademacher random variables. 
\end{definition}

\begin{lemma}[Properties of Type-2 Constant \cite{Led-Tal-book,tomczak1989banach}]
\label{type2-prop}
\begin{enumerate}
\item Let $A$ be an invertible linear map. \\ If $\norm{x}_D:= \norm{A^{-1}x}_K$ for all $x \in X$ then $T_2(X,\norm{.}_D)=T_2(X,\norm{.}_K)$ 
 \item Every Euclidean norm has type-2 constant 1.
\item If $Y$ is a subspace of $X$, then $T_2(Y) \leq T_2(X)$.
		
		\item If $X$ is $n$-dimensional, then $T_2(X) \leq \sqrt{n}$, and $\ell_1$-norm has type-2 constant $\sqrt{n}$.
		\item Let $2 \leq p < \infty$. Then, $T_2(\ell_p^n) \lesssim \sqrt{\min\{p,\log n\}}$, where $\ell_p^n = (\mathbb R^n, \|\cdot\|_p)$. 
	\end{enumerate}

\end{lemma}
\section{Approximate Rank Estimate via Energy Increment} \label{section:Energy}
Energy increment is a general strategy in additive combinatorics to set up a greedy approximation to an a priori unknown object, see \cite{Tao-RS}. Our theorems and algorithms in this section are inspired by the energy increment method as we explain below. We begin by presenting an approximate rank estimate for $L_r$-norms.
\begin{theorem} \label{thm:intro1}
For $r\in[2,\infty]$, $\norm{.}_r$ denotes the $L_r$-norm on $P_{n,d}$. Then, for any $f \in P_{n,d}$ and $\varepsilon > 0$ we have
\[ 
    { \rm srank}_{\norm{\cdot}_r, \varepsilon}(f) \leq \frac{\norm{f}_{\rm HS}^2}{\varepsilon^2},
\]
where $\norm{\cdot}_{\rm HS}$ denotes the Hilbert-Schmidt norm.
\end{theorem} 
One may wonder why this result is interesting for all $L_r$-norms when it takes the strongest form for $r=\infty$. The reason is, of course, the computational complexity. Symmetric tensors that are close to each other in terms of $L_{\infty}$-distance behave almost identical as homogeneous functions on $S^{n-1}$, but it is NP-Hard to compute $L_{\infty}$-distance for $d \geq 4$. For $r > n \log(ed)$ the norms $L_r$ and $L_{\infty}$ on $P_{n,d}$ are equivalent, see Lemma \ref{lem:Barv-ineq}. So we only hope to be able to compute approximate decomposition for $L_r$ where $r$ is not proportional to $n$. Algorithm \ref{alg:ga} and Theorem \ref{thm:alg-greedy} below delineate the trade-off between the tightness of the estimate depending on $r$ and the cost of computation.

Now  we present our energy increment algorithm. In Algorithm \ref{alg:ga}, $\Pi_W$ denotes the orthogonal projection on  the subspace $W$ with respect to the Hilbert-Schmidt norm, and $q_v :=  \underbrace{v \otimes v \otimes \cdots \otimes v}_\text{d times}$. 
\begin{algorithm} [H]
	\caption{Approximate Rank via Energy Increment} \label{alg:ga}
	\begin{algorithmic}[1]
	\STATE{\bfseries Input} $f\in P_{n,d}$, $\norm{.}_r$ for $2 \leq r \leq \infty$, and  $\varepsilon >0$.
	\STATE Initialize $\tilde{f}=0$, $e=\infty$, $W=\{ 0 \}$.
	\REPEAT
	\STATE Find a $v \in S^{n-1}$ such that $\frac{1}{2} \norm{f-\tilde{f}}_r \leq \abs{(f-\tilde{f})(v)}$.
	\STATE $W = {\rm span} ( W \cup \{q_v \})$  \\
	$\tilde{f} = \Pi_{W} (f)$ ,  $e = \norm{f-\tilde{f}}_r$ \\
	\UNTIL{$e < \varepsilon$}
	\STATE {\bfseries Output} $\tilde{f}$ 
    \STATE {\bf Post-condition} $\|f-\tilde f\|_r <\varepsilon$, and ${\rm srank}(\tilde f) \leq \varepsilon^{-2}\|f\|_{\rm HS}^2$
	\end{algorithmic}
\end{algorithm}

Details on the implementation of steps in Algorithm \ref{alg:ga} are explained in Section \ref{experiments} alongside some experimental results. Our next theorem gives a sampling approach for the search step (4).
\begin{theorem} \label{sample-r}
Let $n,d\geq 1$ and $2\leq r\leq n\log(ed)$. Let $p \in P_{n,d}$ and suppose $v_1,v_2,\ldots,v_N$ are vectors that are sampled independently from the uniform probability measure on the sphere $S^{n-1}$. Then, we have
\[ \mathbb{P} \left( \max_{i \leq N} \abs{p(v_i)} \geq \frac{1}{2} \norm{p}_r \right)  \geq 1 - \exp \left( -N /[ \alpha(n,d,r)]^{2r}\right),
\] where $\alpha(n,d,r) := \min \{ (c_1r)^{d/2} , {\binom{rd + n-1}{rd}}^{\frac{1}{2r}} \}$ for a constant $c_1$. In particular, if $N \geq t [\alpha(n,d,r)]^{2r}$ we have 
\[ \mathbb{P} \left( \max_{i \leq N} \abs{p(v_i)} \geq \frac{1}{2} \norm{p}_r \right)  \geq 1 - e^{-t}.
\]
\end{theorem}
The proof of Theorem \ref{sample-r} is included in section \ref{samplesize}. As a consequence of Theorem \ref{sample-r} and the bounds obtained in the proof of Theorem \ref{thm:intro1}  we have the following result on Algorithm \ref{alg:ga}.
\begin{theorem} \label{thm:alg-greedy}
For a given $f \in P_{n,d}$ and $r \in [2,\infty]$, 
\begin{itemize}
    \item Algorithm \ref{alg:ga} takes at most $\frac{\norm{f}_{\rm HS}^2}{\varepsilon^2}$ many loops before terminating;
    \item for step (4) in Algorithm \ref{alg:ga}: searching over a uniform sample on $S^{n-1}$ with size $N \geq t [ \alpha(n,d,r)]^{2r}$, where $\alpha(n,d,r)$ as in Theorem \ref{sample-r}, yields a point $v \in S^{n-1}$ such that $\frac{1}{2}\norm{f}_r \leq \abs{f(v)}$ with probability at least $1 - e^{-t}$;
    \item the output $\tilde{f} $ of Algorithm \ref{alg:ga} satisfies the following properties:
    \[ 
        \norm{f-\tilde{f}}_r \leq \varepsilon \; , \; {\rm srank}(\tilde{f}) \leq \#\{ \text{\rm loops before termination of Algorithm 1}\}\leq \frac{\norm{f}_{\rm HS}^2}{\varepsilon^2}. 
    \]
\end{itemize}
\end{theorem}

\subsection{Upper bound for the number of steps in Algorithm 1}
The energy increment method gives a general strategy to set up a greedy procedure to decompose a given object into ``structured'', ``pseudorandom'', and ``error'' parts \cite{Tao-RS,lovasz-szemeredi}. In what follows we apply this strategy to obtain a low-rank approximation for a symmetric tensor. 
\begin{lemma} [Greedy Approximation] \label{lem:ga}
	Let $(H, \langle \cdot, \cdot \rangle)$ be an inner product space, $\tau : H \to [0,\infty)$ a cost function, and suppose $S \subset B_H=\{ z \in H : \|z\|_H^2 = \langle z,z\rangle =1\}$ separates points in $H$ with respect to $\tau$, that is, 
	 \begin{align*}
	 	\tau( f ) \leq \sup_{w\in S} | \langle f, w \rangle | , \quad f\in H.
	 \end{align*}
	 Then, given $f\in H$ and $\varepsilon>0$ there exist $m$ points $w_1, \ldots, w_m\in S$ with $m\leq \lfloor \|f\|_H^2/\varepsilon^2 \rfloor$ and scalars $\lambda_1, \ldots, \lambda_m$ such that 
	 \begin{align*}
	 	\tau \left ( f - \sum_{i=1}^m \lambda_i w_i \right) \leq \varepsilon.
	 \end{align*}
\end{lemma}
\begin{proof}[Proof of Lemma \ref{lem:ga}]
 To begin with we assume that for the given $f\in H$ and $\varepsilon>0$ we have $\tau(f) >\varepsilon$. Then, by
the separation property there exists $w_1\in S$ so that $|\langle f, w_1 \rangle| > \varepsilon$. Now, let $W_1:= {\rm span}\{w_1\}$, $p_1:= P_{W_1}(f)$ be the orthogonal projection of $f$ onto $W_1$, and note that
\begin{align*}
	\varepsilon < | \langle w_1, f \rangle| = |\langle w_1, p_1 \rangle | \leq \|p_1\|_H.
\end{align*}
If $\tau(f - p_1) \leq \varepsilon$ the process stops. If $\tau(f-p_1) >\varepsilon$ then, by the separation property again, there exists
$w_2\in S$ so that 
\begin{align*}
	\varepsilon < |\langle f-p_1 , w_2 \rangle| = |\langle p_2 -p_1, w_2 \rangle| \leq \|p_2-p_1\|_H, 
\end{align*} where $p_2:=P_{W_2}(f)$ and $W_2:={\rm span} \{w_1, w_2\}$. If $\tau(f-p_2)  \leq \varepsilon$ the process
stops. If $\tau(f- p_2) >\varepsilon$ we repeat. After $m$ steps we have
extracted $w_1, \ldots w_m\in S$, built the flag of finite-dimensional subspaces
\begin{align*}
 	& \{ {\bf 0} \}=W_0 \subset W_1 \subset \ldots \subset W_m  \\
 	& W_s={\rm span}\{w_1, \ldots, w_s\}, \, s=1,\ldots,m,
\end{align*} and 
the lattice of their corresponding orthogonal projections $P_{W_s}, \, s=1, \ldots, m$ with $\|p_s-p_{s-1}\|_H >\varepsilon$, where 
$p_s=P_{W_s}(f)$ for $s=1, \ldots, m$ (here $p_0=P_{W_0}={\bf 0}$).

\smallskip

\noindent {\it Claim.} This process terminates after at most $m$ steps where $m<\|f\|_H^2/\varepsilon^2$, that is $\tau(f-p_m) \leq \varepsilon$.

\smallskip

\noindent {\it Proof of Claim.} Indeed; we may write
\begin{equation*}
    \|f\|^2_H 
    \geq 
    \|p_m\|_H^2  
    = 
    \left\| \sum_{s=1}^m (p_s - p_{s-1}) \right\|_H^2 
    = 
    \sum_{s=1}^m \|p_s-p_{s-1}\|_H^2,
\end{equation*} 
where we have used that $\langle p_k-p_{k-1}, p_{\ell} - p_{\ell-1} \rangle=0$ for $k < \ell$. Since $\|p_s-p_{s-1}\|_H >\varepsilon$ the claim is proved. 
To complete the proof of the lemma notice that $p_m\in W_m$, hence $p_m = \sum_{i=1}^m \lambda_i w_i$ 
for some scalars $\lambda_1, \ldots, \lambda_m$.
\end{proof}
The intuition suggested by the lemma is easy to express: As long as one uses a cost function $\tau$ that is upper bounded by $\sup_{w\in S} | \langle f, w \rangle |$, Lemma \ref{lem:ga} gives a greedy approximation to input object $f$ with controlled distance in terms of the cost $\tau$.
\begin{proof}[Proof of Theorem~\ref{thm:intro1}]
We use the set $S:=\{ \underbrace{v \otimes v \otimes v \otimes \cdots \otimes v}_\text{d times} : v \in S^{n-1}\}$, the inner product $\langle . , . \rangle_{\rm HS}$, and the cost function $\norm{.}_r$ to set up the greedy approximation outlined in Lemma \ref{lem:ga}. The proof relies on the following observations: 
\begin{enumerate}
    \item $\norm{g}_r \leq \norm{g}_{\infty} = \sup_{q \in S} \abs{\langle g , q \rangle_{\rm HS}}$ for all $g \in P_{n,d}$ and all $2 \leq r \leq \infty$,
    \item if one follows the proof of Lemma \ref{lem:ga} applied to our specific case, one observes that $w_i = v_i \otimes v_i \cdots \otimes v_i$ for some $v_i \in S^{n-1}$. 
\end{enumerate}
So, ${\rm srank}(\sum_{i=1}^m \lambda_i w_i ) \leq m \leq \frac{\norm{f}_{\rm HS}^2}{\varepsilon^2}$. 
\end{proof}
\subsection{Bounds on the sample size for executing the step (4) in Algorithm 1} \label{samplesize}

This section is to prove Theorem \ref{sample-r}. We start with proving a reverse H\"older inequality for symmetric tensors.

\begin{lemma} \label{lem:revH}
	Let $p\in P_{n,d}$, then for $n \geq 2d$ and $k\in[2,n/d]$ we have 
	\begin{align*}
	\norm{p}_k \leq (Ck)^{d/2}  \norm{p}_2 ,
	\end{align*} 
	where $C>0$ is an absolute constant.
\end{lemma}

\begin{proof}[Proof of Lemma \ref{lem:revH}] Let $Z\sim N({\bf 0}, I_n)$ be a standard Gaussian vector in $\mathbb R^n$. We will make use of the following facts:
\begin{fact} 
$Z/\|Z\|_2$ is uniformly distributed on $S^{n-1}$ and $\|Z\|_2$ is independent of $Z/ \|Z\|_2$. Thereby, for $r>0$, it follows that 
	\begin{align*}
		\mathbb E |p(Z)|^r = \mathbb E\|Z\|_2^{rd} \cdot \| p \|_{L_r}^r.
	\end{align*} 
\end{fact}
For a proof the reader is referred to \cite{SZ}. The next fact is a consequence of the Gaussian hypercontractivity, see e.g., \cite[Proposition 5.48.]{ASz}. 
%{\color{red} Could you please make the citation precise, in the sense that see Theorem X from ...} 
\begin{fact} For any tensor $Q$ of degree at most $d$ and for every $r\geq 2$ one has 	
	\begin{align*}
		\left( \mathbb E |Q(Z)|^r\right)^{1/r} \leq (r-1)^{d/2} \left( \mathbb E|Q(Z)|^2 \right)^{1/2}.
	\end{align*} 
\end{fact}
Finally, we need the asymptotic behavior of high-moments of $\|Z\|_2$.
\begin{fact}
	For $r>0$ we have $\mathbb E\|Z\|_2^r = 2^{r/2} \Gamma(\frac{n+r}{2}) / \Gamma(\frac{n}{2})$. 
	This follows by switching to polar coordinates.	Therefore, for $r>0$, Stirling's approximation
	yields
	\begin{align*}
		\left( \mathbb E \|Z\|_2^r \right)^{1/r} \asymp \sqrt{n+r}.
	\end{align*}
\end{fact} 
Finally, taking into account the above facts, we may write
\begin{align*}
	\|p \|_k^k = \frac{\mathbb E|p(Z)|^k }{\mathbb E \|Z\|_2^{kd}} \leq \frac{ (k-1)^{\frac{kd}{2} } (\mathbb E|p(Z)|^2)^{k/2} }{\mathbb E \|Z\|_2^{kd} }
					\leq (k-1)^{\frac{dk}{2} } \|p\|_2^k\frac{ \left ( \mathbb E\|Z\|_2^{2d} \right)^{k/2} }{ \mathbb E \|Z\|_2^{kd}}.
\end{align*} 
Using the estimate for the moments of $\|Z\|_2$ we obtain
\begin{align*}
	\|p\|_k \leq (Ck)^{d/2} \left( \frac{n+2d}{n+kd} \right)^{d/2} \|p\|_2,
\end{align*} and the result follows.
\end{proof}

\begin{proof}[Proof of Theorem \ref{sample-r}]
First, note that we may write
	\begin{align*}
		\mathbb P \left( \max_{i\leq N} |p(X_i)| < \frac{1}{2} \norm{p}_r \right)
		& = \left [ \mathbb P \left( |p(X_1)| < \frac{1}{2} \norm{p}_r \right) \right ]^N \\
		& = \left[ 1- \mathbb P \left( |p(X_1)| \geq  \frac{1}{2} \norm{p}_r \right) \right ]^N \\
	  & \leq \exp \left (-N \mathbb P \left( |p(X_1)| \geq  \frac{1}{2} \norm{p}_r \right) \right ).
	\end{align*}
Second, we provide a lower bound for the probability $\mathbb P \left( |p(X_1)| \geq  \frac{1}{2} \|p\|_r \right)$. By the Paley-Zygmund inequality 
we obtain 
\[
	\mathbb P \left(  \abs{p(X_1)} \geq  \frac{1}{2} \norm{p}_r \right) \geq (1- 2^{-r})^2 \frac{\norm{p}_r^{2r}}{\norm{p}_{2r}^{2r}}.
\]
To bound the ratio $ \norm{p}_{2r} / \norm{p}_r$, we employ Lemma \ref{lem:Barv-ineq} and Lemma \ref{lem:revH} as follows:
\[ \norm{p}_{2r} \leq (C_1r)^{d/2} \norm{p}_2 \leq (C_1r)^{d/2} \norm{p}_r \]
so
\[	\norm{p}_{2r} \leq \norm{p}_\infty \leq {\binom{rd + n-1}{rd}}^{\frac{1}{2r}} \norm{p}_r. \] 
Therefore,
\[
	\frac{\norm{p}_{2r} }{ \norm{p}_r } \leq  \min \{ (c_1r)^{d/2} , {\binom{rd + n-1}{rd}}^{\frac{1}{2r}} \},
\]
which completes the proof.
\end{proof}

\subsection{Comparison with earlier results and open questions}
First, let us write a consequence of Theorem  \ref{thm:intro1} for an easier interpretation.
\begin{corollary}
For $r\in[2,\infty]$, $\norm{.}_r$ denotes the $L_r$-norm on $P_{n,d}$. Then, for any $f \in P_{n,d}$ and for any $0 < \delta < 1$ there exists a $q \in P_{n,d}$ with $\norm{f-q}_r \leq \frac{ \norm{f}_{\rm HS}}{(1-\delta) \sqrt{n}}$ and $ { \rm srank}(q) \leq n (1-\delta)^2$.
\end{corollary}
To bring this result to its most simple form: For the case of symmetric matrices and operator norm, i.e., $d=2$ and $r=\infty$, this result says that the closest singular matrix w.r.t. to the operator norm is at most $\frac{\norm{f}_{\rm HS}}{\sqrt{n}}$ away. So, in this very special case the result seems to be tight; once can consider the case where all singular values of $f$ are equal and use the Eckart-Young theorem. However, for general tensor spaces equipped with $L_r$-norms, for moderately small $r$, the result does not seem to be tight. The following problem remains open: 
\begin{problem}
Obtain sharp estimates on the approximate symmetric rank with respect to all $L_r$-norms for $r \in [2,\infty)$ and for all $P_{n,d}$.
\end{problem}

The main of result of \cite{blekherman-teitler} combined with the celebrated Alexander-Hirschowitz Theorem, see e.g. \cite{ottaviani-expose}, provides a bound for the ${\rm srank}$ of real symmetric tensors. In particular, the ${\rm srank}$ is typically between  $\frac{1}{n} \binom{n+d-1}{d}$ and $\frac{2}{n} \binom{n+d-1}{d}$ for $d >2$ except for the cases $(n,d)\in \{(3,4), (4,4), (5,4), (5,3)\}$.  This beautiful result coming from algebraic geometry is exact, static and it universally holds for any symmetric $d$-tensor. Our estimate in Theorem \ref{thm:intro1}, and later in  Theorem \ref{thm:intro2}, are  approximate, dynamic, and give a different estimate depending on the norm of the input. This basically shows that the symmetric rank and approximate symmetric rank are different in nature. Note that we fix $\varepsilon >0$, that is, the approximate rank notion is also different from the rank notions that require taking limits.

If one is still interested in strict comparison, Theorem \ref{thm:intro1} improves upon the algebraic geometry estimate for
\[ \ln (\frac{1}{\varepsilon}) <  \frac{1}{2} \ln \binom{n+d-1}{d}  - \ln \norm{f}_{\rm HS} - \frac{1}{2}\ln n  \leq   \frac{d-1}{2} \ln n - \ln \norm{f}_{\rm HS} 
   \]
So, for Theorem \ref{thm:intro1} to be useful for small $\varepsilon$ we need  $\ln \norm{f}_{\rm HS}$ to be smaller compared to $\frac{d-1}{2} \ln n$. As a rule of thumb, we need $\ln \norm{f}_{\rm HS} \leq \frac{d}{4} \ln n$, and the smaller is the better.
To see if this is meaningful for applications, we looked at input models for symmetric tensors that are considered in recent literature. As an example, in \cite{slpjoe} the input model for tensors is the following:  One samples $a_1,a_2,\ldots,a_K \in S^{n-1}$ where $K=O(n^{\frac{d}{2}})$ in a way that makes the collection of rank one tensors $a_i \otimes a_i \otimes \ldots \otimes a_i$ to have ``restricted isometry property''. Then one considers the tensor $p:=\sum_{i=1}^K a_i \otimes a_i \otimes \ldots \otimes a_i$ and adds a small perturbations to it. That is, one consider $f:=p+h$ where $h$ has very small norm, e.g., $\norm{h}_{\rm HS}=O(\frac{1}{n})$. Due to the ``restricted isometry property" one has 
\[ \norm{p}_{\rm HS}^2 \sim \sum_{i=1}^k \norm{a_i \otimes a_i \otimes \ldots a_i}_{\rm HS}^2 = K . \]
In the end, the input tensor $f$ has $\norm{f}_{\rm HS}=O(n^{\frac{d}{4}})$, and $f$ is $O(\frac{1}{n})$ close to a tensor $p$  with rank $O(n^{\frac{d}{2}})$. A main result in \cite[Theorem 16]{slpjoe} is to show that the proposed algorithm (with high probability) removes the ``noise'' in $f$, and recovers the decomposition with rank $O(n^{\frac{d}{2}})$. Here we will consider a much more flexible input model and still obtain a similar result: Let $q \in P_{n,d}$ be a symmetric tensor with $\norm{q}_{\rm HS}=O(n^{\frac{d}{4}})$. We impose no further assumptions on $q$. For instance, if $q$ is a typical input then it has symmetric rank between $\frac{1}{n} \binom{n+d-1}{d}$ and $\frac{2}{n} \binom{n+d-1}{d}$, that is, for a typical $q$ we have ${\rm srank}(q) = \Omega(n^{d-1})$. For this  input tensor $q$, Theorem \ref{thm:intro1} yields the following: For any $\varepsilon >0$
\[  { \rm srank}_{\norm{\cdot}_r, \varepsilon}(q) \leq \frac{ O(n^{\frac{d}{2}})}{\varepsilon^2}. \]
The meaning of this is that for a fixed small $\varepsilon$, say $\frac{1}{\varepsilon}=\ln n$, Theorem \ref{thm:intro1} (and Algorithm \ref{alg:ga}) finds an $\varepsilon$-close symmetric tensor that has rank $O(n^{\frac{d}{2}} \ln^2 n)$.  

The usage of random tensors as a testing ground also brings the following problem, which remains open to the best of our knowledge.

\begin{problem}
    Let $f$ be an isotropic Gaussian, with respect to the inner product $\langle . , . \rangle_{\rm HS}$, random element of the vector space $P_{n,d}$, and let $\varepsilon >0$ be fixed. Prove upper and lower bounds, that holds with high probability, for the quantity ${ \rm srank}_{\norm{\cdot}_r, \varepsilon}(f)$ in the range $r \in [2, \infty)$. %{\color{red} P: What is the randomness? How is the sub-Gaussian tensor defined?}
\end{problem}

The development in this paper is entirely self-contained. Our search to locate earlier appearance of a similar results in the literature yielded only the following. The main result of \cite{vempala-tensor} used for the specific case of symmetric tensors corresponds to our Theorem \ref{thm:intro1} for $r=\infty$: Computing with $L_{\infty}$ is generally intractable, but this nice result was sufficient for the theoretical purposes the authors considered. Our contribution is to prove algorithmic results that hold for all $L_r$-norms:  Algorithm \ref{alg:ga} and Theorem \ref{thm:alg-greedy} delineate the trade-off between the computational complexity (the sample size) and the tightness of approximation for the entire range $r \in [2, \infty]$.

There is a vast literature on tensor decomposition algorithms. We do not intend to survey this vast and  interesting literature due to following reason: Our Algorithm 1 and existing tensor decomposition algorithms in literature have different goals. The goal of Algorithm 1 is to show that the approximation in Theorem \ref{thm:intro1} is efficiently computable as long as the used $L_r$-norm is efficiently computable, i.e. $r$ is small and independent of $n$. Existing algorithms on symmetric tensor decomposition aims to solve a much harder problem, that is, to find an optimal low-rank approximation for a given symmetric tensor. This requires finding the ``latent'' rank-one tensors, and is known to be a hard problem \cite{hillar2013most,limsymmetric}. We do not aim to solve this NP-Hard problem:  Our algorithm only gives an upper bound for the approximate rank.  Practically, Algorithm \ref{alg:ga} can be used to pre-process a given tensor before deploying a more expensive tensor decomposition algorithm: Most tensor decomposition algorithms require a guess on the rank of the input tensor, for which the guaranteed rank upper bound from Algorithm \ref{alg:ga} can be used.

\subsection{Implementation of Algorithm 1} \label{experiments}
We note from the outset that our current implementation is in a preliminary form. Our main goal is to show that the approximate rank estimate in Theorem 3.1 is constructive: a decomposition that realizes the estimate is effectively computable. We do not claim to have a scalable implementation.

We used  a Windows 11 PC, with a Intel Core i7 2.3 GHz processor and 32.0 GB installed ram to experiment with the implmentation. The code is available on first authors' personal webpage.
\begin{enumerate}
    \item We computed $L_r$-norms by (re)implementing  (with Cristancho and Velasco's kind permission) the quadrature rules from \cite{mauricio} in Python. 
    The quadrature rule for computing the $L_r$-norms is by far the most expensive step of the algorithm.
    
    \item Theorem~\ref{sample-r} provides a bound on the sample size for step (4). In practice, as long as one finds a vector that satisfies the requirement in step 4 of Algorithm \ref{alg:ga} the computation is correct. For experiments we fixed a sample size of $100,000$ and loop in case a vector with such characteristics is not found. We observe that even with this fixed sample size a vector with the correct characteristics was always found. 
    
    \item A practical improvement for Algorithm \ref{alg:ga} came from the following observation: In the implementation we put the extra constraint of the new vectors for step 4 should have an angle bigger than $\arccos(0.8)$ with the older ones. This practical trick observably improved the performance. In future work, this idea needs to be improved and analyzed.
    
    \item For the experiment, we consider a randomly generated $n$-variate $2d$-tensors of the type 
    \[ f =
    \sum_{i=1}^m c_i \; \underbrace{v_i \otimes v_i \otimes \cdots \otimes v_i}_\text{2d times} + \frac{\varepsilon}{2} \sum_{i_1,i_2,\ldots,i_d} e_{i_1} \otimes e_{i_1} \otimes e_{i_2} \otimes e_{i_2} \otimes \ldots \otimes e_{i_d} \otimes e_{i_d}  
    \]
where $c_1,\ldots,c_m\in \mathbb R$ uniformly distributed according to a standard Gaussian, and $v_1, \ldots, v_m$ are uniformly distributed on the $n$-dimensional sphere. Basically, the input $f$ is a very high rank symmetric tensor that is $\frac{\varepsilon}{2}$-close to a rank $m$ tensor. We get the following results for different values of $m,n,d,r,\varepsilon$ in the experiment:
\begin{itemize}
    \item For $m=10$, $n=4$, $2d=4$, $r=4$, $\varepsilon=0.3$,  the dimension of the space is $35$, and the algorithm found an $\tilde{f}$ of rank $3$ for which $\norm{f-\tilde{f}}_r < 0.29$ in 3.43 seconds. 
    
    \item For $m=10$, $n=4$, $2d=24$, $r=4$, and $\varepsilon =0.3$, the dimension of the space is $2925$, and the algorithm found an $\tilde{f}$ of rank $2$ for which $\norm{f-\tilde{f}}_r < 0.21$ in about 4.8 seconds.

    \item For $m=10$, $n=6$, $2d=18$, $r=4$, and $\varepsilon =0.3$, the dimension of the space is  33649, and the algorithm an $\tilde{f}$ of rank $1$ for which $\norm{f-\tilde{f}}_r < 0.22$ in about 2 min 48 seconds.
    
    \item For $m=10$, $n=8$, $2d=8$, $r=8$, and $\varepsilon =0.3$, the dimension of the search space is $6435$, and the algorithm found an $\tilde{f}$ of rank $4$ for which $\norm{f-\tilde{f}}_r < 0.29$ in about 6 min 57 seconds.

    \item For $m=14$, $n=12$, $2d=10$, $r=8$, and $\varepsilon =0.3$  we were not able to run the algorithm due to quadrature rule taking too much space in memory.
\end{itemize}

\end{enumerate}
Our experiment enforced our belief that Algorithm 1 is as efficient as the quadrature rule to compute the $L_r$-norm; the rest of the steps do not create much computational overhead. This is evident from the sensitivity of the computing time to the number of variables rather than the degree of the tensor: the size of the quadrature nodes grows moderately with respect to degree but drastically with respect to variables. A more optimized implementation of the quadrature rule, or a parallellized version, would greatly improve the performance and allow computations with more variables.

\section{Approximate Rank Estimate via Sparsification} \label{section:Maurey}
\begin{algorithm}  \caption{Approximate Rank via Sparsification}
\begin{algorithmic}[1] \label{alg:empirical-approx}
		\STATE {\bfseries Input} $p(x) = \sum_{i=1}^N c_i  v_i \otimes v_i \cdots \otimes v_i$, $T=T_2(P_{n,d}, \norm{.})$, and $\varepsilon >0$.
		\STATE $\mu$ is the measure supported on set $\{ 1, 2, \ldots, N\}$ with $\mu(i)=\frac{\abs{c_i}}{\sum_{i} \abs{c_i}}$
		\STATE Sample $k:= \lceil 4\varepsilon^{-2} T^2 (\sum_{i=1}^N \abs{c_i})^2\rceil$ many elements $\lambda_1,\lambda_2,\ldots, \lambda_k$ from  $\mu$. 
		\STATE Set $q_k:= \frac{1}{k} \sum_{i=1}^k {\rm sign}(c_{\lambda_i})v_i \otimes v_i \cdots \otimes v_i$
		\IF{$\norm{p-q_k}>\varepsilon$}
		\STATE Return to step 3
		\ELSE
		\STATE Set $q=q_k$
		\ENDIF
		\STATE {\bfseries Output} $q$ 
		\STATE {\bfseries Post-condition} $\norm{p-q} \leq \varepsilon$ , and
		${\rm srank}(q) \leq 4  \varepsilon^{-2} T^2 (\sum \abs{c_i})^2$

\end{algorithmic}
\end{algorithm}
Algorithms and theorems in this section rely on Maurey's empirical method from geometric functional analysis which was presented in the 80s paper by \cite{Pis-1}.  Special cases of this lemma have been (re)discovered many times in recent literature, e.g., \cite{Bar,Ivan} where further algorithmic results were also obtained. We reproduce Maurey's idea in section \ref{empirical} for expository purposes. Note that the type-2 constant, $T_2$, was defined in Section \ref{type-2}. 
\begin{theorem} \label{thm:intro2}
Let $\norm{\cdot}$ be a norm on $P_{n,d}$ such that $\norm{v \otimes v \otimes \ldots \otimes v} \leq 1$ for all $v \in S^{n-1}$. Let $T$ denote type-2 constant of $P_{n,d}, \norm{.}$, let  $\norm{\cdot}_*$ denote the nuclear norm. Then, for any $f \in P_{n,d}$ and $\varepsilon >0$ we have
\[
    { \rm srank}_{\|\cdot\|, \varepsilon}(f) \leq \frac{4T^2 \norm{f}_{*}^2}{\varepsilon^2} . 
\]
\end{theorem}

Algorithm 2  admits any decomposition as an input and gives a low-rank approximation via sparsification. In the specific case of the input being a  nuclear decomposition, the algorithm finds an approximation that is a realization of  Theorem \ref{thm:intro2}.
\begin{theorem} \label{maurey-step-count}
Algorithm \ref{alg:empirical-approx} terminates in $\ell$ steps with a probability of at least $1-2^{-2\ell}$. 
\end{theorem} 
Theorem \ref{thm:intro2} and Theorem \ref{maurey-step-count} are proved in Section \ref{empirical}.
\subsection{Sparsification via Maurey's Empirical Method} \label{empirical}

\begin{lemma} [Empirical Approximation] \label{empirical-approximation}
	Let $(X, \|\cdot\|)$ be a normed space and a set $S\subset B_X:= \{x\in X : \|x\| \leq 1\}$.
	For any $x\in {\rm conv} S $ and $m\in \mathbb N$, there exist $z_1, \ldots, z_m$ in $S$ (not necessarily distinct) such that 
	\[
		\left\| x- \frac{1}{m} \sum_{j=1}^m z_j \right\| \leq \frac{2 T_2(X)}{\sqrt{m}}.
	\]
\end{lemma}
\begin{proof}
Since $x\in {\rm conv} S$ there exist $v_1, \ldots, v_\ell \in S$ and $\lambda_1, \ldots, \lambda_\ell \in [0,1]$ with 
$\lambda_1+\cdots +\lambda_\ell=1$ and $x= \lambda_1v_1+\cdots +\lambda_\ell v_\ell$. We introduce the random vector $Z$ taking values on $\{v_1, \ldots, v_\ell\}$ with probability
distribution $\mathbb{P}$ where $\mathbb P(Z=v_i) = \lambda_i$ for $i=1,2, \ldots, \ell$. Clearly, $\mathbb E [Z] =x$. Now we apply an empirical approximation of $\mathbb E [Z]$
in the norm $\|\cdot\|$. To this end, let $Z_1, \ldots,Z_m$ be a sample, that is, $Z_i$ are independent copies of $Z$. We set 
$Y_m : =\frac{1}{m} \sum_{j=1}^m Z_j$ and note that $\mathbb E[Y_m]=\mathbb E[Z] =x$. Now we use a symmetrization argument: introduce $Z_i'$ independent copies of $Z_i$, whence 
$\mathbb E[Y_m']= \mathbb E[\frac{1}{m}\sum_{i=1}^m Z_i']=x$. Thus, by Jensen's inequality we readily get
    \[
    \mathbb E \|Y_m -  x \|^2 = \mathbb E \|Y_m -  \mathbb{E} \; Y_m^{'} \|^2 \leq \mathbb E  \|Y_m-Y_m'\|^2 = \frac{1}{m^2} \mathbb E \left \|\sum_{j=1}^m (Z_j -Z_j') \right\|^2.
    \]
Next, $Z_i-Z_i'$ are symmetric, whence, if $(\varepsilon_i)$ are independent Rademacher random variables, and independent from both $Z_i,Z_i'$, then the joint distribution of $\varepsilon_i(Z_i-Z_i')$ is the same with $(Z_i-Z_i')$. Thereby, we may write
	%\begin{multline*}
	 \[
	 \frac{1}{m^2} \mathbb E \left \|\sum_{j=1}^m (Z_j -Z_j') \right\|^2  = \frac{1}{m^2} \mathbb E \left \|\sum_{j=1}^m  \varepsilon_j (Z_j-Z_j') \right\|^2 \leq \frac{4}{m^2} \mathbb E \left \|\sum_{j=1}^m \varepsilon_j Z_j \right\|^2    
	 \]
	%\end{multline*}
%{\color{red} @ Petros: can we somehow explain the first inequality more?}
where in the last passage we have applied the triangle inequality and the numerical inequality $(a+b)^2\leq 2(a^2+b^2)$. Using  the definition of the type-2 constant, we have
$\mathbb E \left \|\sum_{j=1}^m \varepsilon_j Z_j \right\|^2 \leq T^2 \sum_{j=1}^m \|Z_j\|^2 \leq mT^2$,  where we have used the fact that $\| Z_j \|\leq 1$ a.s. The result follows from the first-moment method. 
\end{proof}

\begin{proof}[Proof of Theorem \ref{thm:intro2} ]
 Let $p \in P_{n,d}$ with $p \neq 0$ and set $p_1 := p / \norm{p}_{*}$. Since the nuclear norm is induced by the convex body $V_{n,d}$, we have that $p_1 \in V_{n,d}$. Hence, by Lemma \ref{empirical-approximation} we infer that there exist $v_i \in S^{n-1}$ for $i=1,2,\ldots,m$ , $m = \ceil{\frac{4T^2 \norm{p}_{*}^2}{\varepsilon^2}}$, and $\xi_i \in \{ -1, 1 \}$ such that 
 $  \left\| p_1 - \frac{1}{m} \sum_{i=1}^m \xi_i p_{v_i} 
 \right \| \leq \frac{\varepsilon}{\norm{p}_{*}} ,
 $
 which completes the proof.
\end{proof}

\begin{proof}[Proof of Theorem \ref{maurey-step-count}]
Using the proof of Lemma \ref{empirical-approximation}, it follows that $\mathbb{E} \norm{p-q_k} \leq \frac{\varepsilon}{4}$. Moreover, we also observe that by Markov's inequality $\mathbb{P} \{ \norm{p-q_k} > \varepsilon \} \leq \frac{1}{4}$. Thus, the ``if'' statement at step 5 returns True at the $\ell$-th trial with probability at least $1-2^{-2\ell}$.
\end{proof}

\begin{remark} \label{rem:ea-2}\begin{enumerate}
\item Aiming for better guarantees, i.e., a higher probability estimate of the desired event, one may work with higher moments and apply Kahane-Khintchine inequality. 
\item We should emphasize that the key parameter in the empirical approximation is the ``Radamacher type-2 constant $T_2(S)$ of the set $S$'' rather than the Rademacher type of the ambient space $X$. This simple but crucial observation will permit us to provide tighter bounds in our context (see Theorem \ref{enhancedsparse}). 
\end{enumerate}
\end{remark}

\subsection{Type-2 constant estimates for norms on symmetric tensors}
The results of this section hold for any norm, however, in practice we use the norms that we can efficiently compute. As mentioned earlier, currently our collection of ``efficient norms'' includes the $L_r$ norms thanks to efficient quadrature rules \cite{mauricio}. Our estimates for the type-2 constants of $L_r$-norms on $P_{n,d}$ for $2 \leq r \leq \infty$  is as follows:
\begin{theorem} \label{type-estimate}
	Let $(P_{n,d}, L_r)$ be the space of symmetric $d$-tensors on $\mathbb{R}^n$ equipped with $L_r$-norm as defined in Section \ref{functionalnorms}. Then, for $r\in [2,\infty]$ we have
	\begin{align*}
	T_2(P_{n,d}, L_r) \lesssim \sqrt{ \min \{ r ,  n \log(ed) \} }.
	\end{align*}
\end{theorem}
\begin{proof}[Proof of Theorem \ref{type-estimate}]
Although the fact that $T_2(L_r( \Omega, \mu)) \lesssim \sqrt{r}$ is well-known, see \cite{AK}, we provide here a sketch of proof for reader's convenience. The proof makes use of Khintchine's inequality which reads as follows: Let $\xi_j$ be  independent Rademacher random variables and  $\alpha_j$ be arbitrary real numbers, for $j \in \mathbb{N}$. Then we have
	\[
	\left( \mathbb E \left| \sum_j \alpha_j \xi_j \right|^r \right)^{1/r} \leq B_r \left(\sum_j |\alpha_j|^2 \right)^{1/2},
	\] 
for some scalar $B_r$ with $B_r=O(\sqrt{r})$. Let $h_1, \ldots, h_N \in L_r$, then we may write
	\begin{align*}
	\mathbb E \left \| \sum_{j=1}^N \xi_j h_j \right\|_{L_r}^r & = \int \mathbb E \left| \sum_{j=1}^N \xi_j h_j(\omega) \right|^r \, d\mu(\omega)
													 \leq B_r^r \int \left( \sum_{j=1}^N |h_j(\omega)|^2 \right)^{r/2} \, d\mu(\omega), 
	\end{align*} where we have applied Khintchine's inequality for each fixed $\omega$. Now we recall the following variational argument: for $0<p<1$ and 
	for non-negative numbers $u_1, \ldots, u_N$ one has
	\begin{align*}
		\left(\sum_{j=1}^n u_j^p \right)^{1/p} = 
					\inf \left\{ \sum_{j=1}^N u_j \theta_j : \sum_{j=1}^N \theta_j ^q \leq 1 , \; \theta_j >0 \right\}, \quad q:= \frac{p}{p-1}<0.
	\end{align*}
	Note that, for ``$p=2/r$'' and for ``$u_j = |h_j(\omega)|^r$'', after integration, we have
	\begin{align*}
		\int \left( \sum_{j=1}^N |h_j(\omega)|^2 \right)^{r/2} \, d\mu(\omega) \leq \int \sum_{j=1}^N u_j(\omega) \theta_j \, d\mu(\omega)= 
		\sum_{j=1}^N \theta_j \|h_j\|_{L_r}^r,
	\end{align*} for any choice of positive scalars $\theta_j$ so that $\sum_j \theta_j ^q \leq 1$.

For the type-2 constant of $(P_{n,d}, \|\cdot\|_\infty)$ we combine the type-2 estimate for $L_r$ along with the fact, which follows from Lemma \ref{lem:Barv-ineq}, that 
$c \|\cdot\|_\infty \leq \|\cdot \|_r \leq \|\cdot\|_\infty$ for $r\geq n \log(ed)$.
\end{proof}
\subsection{An improvement of the sparsification estimate}
The definition of the type-2 constant considers all vectors $f_i \in P_{n,d}$ and asks for a constant that satisfies $\mathbb E_{\xi_1,\ldots,\xi_m} \left \| \sum_{i=1}^m \xi_i f_i \right\|^2 \leq T^2 \sum_{i=1}^m \|f_i\|^2$. However, for our sparsification purposes, we only work with vectors of the type $f_i=v \otimes v \otimes \cdots \otimes v$ for some $v \in S^{n-1}$. Instead of using type-2 constant definition, which considers the entire space $P_{n,d}$, if we can re-do our proofs only focusing on the vectors $f_i=v \otimes v \otimes \cdots \otimes v$ we can improve the estimates; see Remark \ref{rem:ea-2}. We obtain such an improvement  for the case of $L_{\infty}$-norm using the following Khintchine type inequality. 
\begin{theorem} [Khintchine inequality for symmetric tensors] \label{thm:Khintch-tensor}
    Let $x_1, \ldots, x_m$ be vectors in $\mathbb R^n$, let $d\in \mathbb N$ and $d\geq 2$, then for any subset $S\subset S^{n-1}$ we have
    \[
        \mathbb E_\varepsilon \sup_{z\in S} \left| \sum_{i=1}^m \varepsilon_i \langle x_i,z\rangle^d \right| \leq 2d 
        \left( \sum_{i=1}^m \|x_i\|_2^{2d} \right)^{1/2} .
    \]
    where $\varepsilon_i$ are independent Rademacher random variables.
    \end{theorem}
 As a consequence of Theorem \ref{thm:Khintch-tensor}, we have:
\begin{theorem}[Improved sparsification for $L_{\infty}$-norm] \label{enhancedsparse}
For $f \in P_{n,d}$ and $\varepsilon >0$, we have
\[
    { \rm srank}_{\norm{.}_{\infty}, \varepsilon}(f) \leq \frac{8  d^2 \; \norm{f}_{*}^2}{\varepsilon^2} . 
\]
\end{theorem}

Observe that if $\tau=v \otimes v \otimes \cdots \otimes v$ for some $v \in \mathbb R^n$ then we have $\norm{v}_2^d=\norm{\tau}_{\infty}$. Also note that for the set $S=S^{n-1}$, we have $\sup_{z\in S}  \left| \sum_{j=1}^m \varepsilon_j \langle x_j,z\rangle^d \right| = \norm{\sum_{j=1}^m \varepsilon_j f_j }_{\infty}$, where $f_i:=  x_i \otimes x_i \otimes \ldots \otimes x_i$ for $i=1,2,\ldots,m$.
%where $v_i \in S^{n-1}$, i.e., $\norm{f_i}_{\infty}=1$, 
Hence, by Theorem \ref{thm:Khintch-tensor} we have 
%{\color{red} Alperen: If anyone wants to explain how we get the bound from Theorem \ref{thm:Khintch-tensor}, that would be great. Since the theorem statement does not have squares on the right hand side whereas the bound below has them, we can get a picky referee poking at it.}
\begin{equation} \label{infty-type}
        \mathbb E \left\|\sum_{i=1}^m \varepsilon_i f_i \right\|_{\infty} \leq 2d \left( \sum_{i=1}^m \norm{f_i}_{\infty}^2  \right)^{1/2}. 
\end{equation}       
Following the proof of Theorem \ref{thm:intro2} line by line, but replacing  the type-2 estimate from Theorem \ref{type-estimate} in the proof with the estimate (\ref{infty-type}), we obtain Theorem \ref{enhancedsparse}, provided that $\|f_i\|_\infty=1$.

\begin{remark}
Theorem \ref{enhancedsparse} improves Theorem \ref{thm:intro2} if $d^2 < n$, which is the common situation when one works with tensors. Theorem \ref{thm:intro2} also immediately improves Step (3) in Algorithm \ref{alg:empirical-approx}: One can use $k\asymp \frac{  d^2 \; \norm{f}_{*}^2}{\varepsilon^2} $ when working with the $L_{\infty}$-norm. 
\end{remark}
\begin{proof}[Proof of Theorem \ref{thm:Khintch-tensor}] 

To ease the exposition we present the argument in two steps:

\medskip

\noindent {\it Step 1: Comparison Principle.} Let $T\subset \mathbb R^m$ 
and $\varphi_j:\mathbb R\to \mathbb R$ be functions that satisfy the Lipschitz condition  
$|\varphi_j(t)-\varphi_j(s)| \leq L_j |t-s|$ for all $t,s\in \mathbb R$ and $\varphi_j(0)=0$ for $j=1,2,\ldots, m$. 
If $\varepsilon_1, \ldots, \varepsilon_m$ are independent Rademacher variables, then 
	\[
		\mathbb E \sup_{t\in T} \left| \sum_{j=1}^m \varepsilon_j \varphi_j(t_j) \right| \leq  2 \mathbb E \sup_{t\in T} \left| \sum_{j=1}^m \varepsilon_j L_j t_j \right|.
	\]
This is consequence of a comparison principle due to Talagrand \cite[Theorem 4.12]{Led-Tal-book}).
Indeed, let $S:= \{(L_jt_j)_{j\leq m} \mid t\in T\}$ and let $h_j(s):= \varphi_j(s/L_j)$. Note that $h_j$ are contractions with $h_j(0)=0$ and 
	\[
		\mathbb E \sup_{t\in T} \left| \sum_{j=1}^m \varepsilon_j \varphi_j(t_j) \right| = \mathbb E \sup_{s\in S} \left | \sum_{j=1} \varepsilon_j h_j(s_j)\right|.
	\] Hence, a direct application of \cite[Theorem 4.12]{Led-Tal-book} yields
	\[
		\mathbb E \sup_{s\in S} \left | \sum_{j=1} \varepsilon_j h_j(s_j)\right| \leq 2 \mathbb E \sup_{s\in S} \left| \sum_{j=1}^m \varepsilon_j s_j\right| 
					= 2 \mathbb E \sup_{t\in T} \left| \sum_{j=1}^m \varepsilon_j L_j t_j\right|,	
	\] as desired.

\medskip

\noindent {\it Step 2: Defining Lipschitz maps.} In view of the previous fact it suffices to define appropriate Lipschitz contractions which will permit us
to further bound the Rademacher average from above by a more computationally tractable average.
To this end, we consider the function $\varphi :\mathbb R\to \mathbb R$ which, for $t\geq 0$, it is defined by
		\[
			\varphi(t) : = \begin{cases} 
								t^d, & 0\leq t \leq 1\\
								d(t-1) + 1, & t\geq 1
								\end{cases},
		\] and we extend to $\mathbb R$ via $\varphi(-t) = (-1)^d \varphi(t)$ for all $t$. 
Note that $f$ satisfies $\|\varphi\|_{\rm Lip} = d$. Now we define $\varphi_j:\mathbb R \to \mathbb R$ by $\varphi_j(t): = \|x_j\|_2^d \varphi(t)$ and notice that 
	$\|\varphi_j\|_{\rm Lip} = d \|x_j\|_2^{d} $. Hence, by the the comparison principle (Step 2) for $T = \{ (\langle z,\bar{ x_j} \rangle)_{ j \leq  m} \mid z\in S^{n-1}\}$, where 
	$\bar{x_j} = x_j/\|x_j\|_2$, we obtain
	\[
		\mathbb E \sup_{z\in S^{n-1}} \left| \sum_j \varepsilon_j \langle z,x_j \rangle^d \right| = 
			\mathbb E \sup_{z\in S^{n-1}} \left| \sum_j \varepsilon_j \varphi_j(\langle z,\overline{x_j}\rangle) \right|
					\leq 2d \mathbb E \sup_{z\in S^{n-1}} \left| \sum_j \varepsilon_j \|x_j\|_2^d \langle z, \overline{x_j} \rangle \right|.
	\] 
Lastly, we have 
	\[
	\mathbb E \sup_{z\in S^{n-1}} \left| \sum_j \varepsilon_j \|x_j\|_2^d \langle z, \overline{x_j} \rangle \right| 
			%=  \mathbb E \sup_{z\in S^{n-1}} \left| \sum_j g_j \|x_j\|_2^{d-1} \langle z, x_j \rangle \right|
				= \mathbb E \left \| \sum_j \varepsilon_j \|x_j\|_2^{d-1} x_j \right \|_2,
	\]
and the result follows by applying the Cauchy-Schwarz inequality and taking into account the fact that $(\varepsilon_j)_{j\leq m}$ are orthonormal in $L_2$. 
\end{proof}

\begin{remark}
Let us point out that if $d\geq 2$ is even, then we may slightly improve the quantity of the datum $(x_i)_{i\leq m}$ on the right
hand-side at the cost of a logarithmic term in dimension. Indeed; let $d=2k$, $k\in \mathbb N$, $k\geq 1$. We apply Step 2 for 
$T=\{ (\langle x_j,\theta\rangle^2 )_{j\leq m} \mid \theta \in S^{n-1} \}$ and the even contractions $\varphi_j:\mathbb R\to \mathbb R$ which, for $s\geq 0$, 
are defined by
$\varphi_i(s) = \min\{ \frac{s^k}{k\|x_i\|_2^{2k-2}}, \frac{\|x_i\|_2^2}{k}\}$. Thus, we obtain
	\[
		\mathbb E \left\| \sum_{i=1}^m \varepsilon_i f_i \right\|_\infty \leq d \mathbb E\left\| \sum_{i=1}^m  \varepsilon_i \|x_i\|_2^{d-2} x_i \otimes x_i \right\|_{\rm op}.
	\]
One may proceed in various ways to bound the latter Rademacher average. 
For example we may employ the matrix Khintchine inequality \cite[Exercise 5.4.13.]{Ver-book} to get
	\[
		\mathbb E\left\| \sum_{i=1}^m  \varepsilon_i \|x_i\|_2^{d-2} x_i \otimes x_i \right\|_{\rm op}	\lesssim 
			\sqrt{\log n} \left \| \sum_{i=1}^m \|x_i\|_2^{2d-2}x_i\otimes x_i \right\|_{\rm op}^{1/2}.
	\] Clearly, $\left\| \sum_{i=1}^m \|x_i\|_2^{2d-2}x_i\otimes x_i \right\|_{\rm op}^{1/2} \leq \left( \sum_{i=1}^m \|x_i\|_2^{2d} \right)^{1/2}$.
\end{remark}
\subsection{Comparison with earlier results and open problems}
The quality of approximation provided by Algorithm 2 depends on the constant $c$ with the property that $c \geq \norm{q}_{*}$. It is known that computing the best such $c$, i.e. the nuclear norm (or the nuclear decomposition), is NP-Hard \cite{friedland2018nuclear}. As mentioned in \cref{nuclear-norm}, one can use sum of squares hierarchy to obtain an increasing sequence of lower bounds for symmetric tensor nuclear norm \cite{nienuclear}. Practically, one would like to have a quick decreasing sequence of upper bounds to compare against the increasing sequence of lower bounds coming from sum of squares hierarchy.
\begin{problem}
    Design an efficient randomized approximation scheme  (approximating from above) for the symmetric tensor nuclear norm. 
\end{problem}

Our search for similar results to Theorem \ref{thm:intro2} in the literature yielded the following:
Theorem 5 of \cite{netzer-approximate} used for symmetric tensors would roughly correspond to the special case of Theorem \ref{thm:intro2} for  Schatten-$p$ norms. The focus of \cite{netzer-approximate} is to demonstrate that separation between different notions of tensor ranks is not robust under perturbation. We work only with ${\rm srank}$ and impose no restrictions on the employed norm. We show that the type-2 constant and the nuclear norm universally govern the quality of the empirical approximation in Algorithm \ref{alg:empirical-approx} for any norm.

%%%%%%%%%%%%%%%%%%%%%%%%%%%%%%%%%%%%%%%%%%%%%%%%%%%%%%
\section{Approximate Rank Estimates via Frank-Wolfe} \label{section:Frank-Wolfe}
%%%%%%%%%%%%%%%%%%%%%%%%%%%%%%%%%%%%%%%%%%%%%%%%%%%%%%
\begin{algorithm}
	\caption{Approximate Rank via Frank-Wolfe}\label{alg:FW}
	\begin{algorithmic}[1]
 
		\REQUIRE
		    $\varepsilon>0$,  a starting point $ q \in V_{n,d}$, and step-size strategy $\gamma_k = \frac{2}{k+1}$. 
            \STATE Initialize $p_0=0$.
		\FOR{$k=0$ to $T-1$}
		    \IF{$\norm{p_k - q}_{\rm HS}  < \varepsilon$}
		        \STATE Halt, set $\tilde{q}:=p_k$,  and output $\tilde{q}$.
		    \ELSE
		        \STATE
		            $h_k = \underset{h \in V_{n,d}}{\argmin} \langle h , \nabla F(p_k) \rangle_{\rm HS}$
		            \STATE 
		            $p_{k+1} = p_k + \gamma_k(h_k - p_k)$
		      \ENDIF
		  \ENDFOR  
    \STATE {\bfseries Post-condition} $\norm{q-\tilde{q}}_{\rm HS} \leq \varepsilon$ and ${\rm srank}(\tilde{q}) \leq 8\varepsilon^{-2}$
	\end{algorithmic}
\end{algorithm}
This section presents a supplementary result for the specific case of using a Euclidean norm in Theorem \ref{thm:intro2}. The theoretical result of this section, \cref{HSrank}, is not stronger than what one could obtain using \cref{thm:intro2}. The main difference is that the corresponding algorithm does not require any decomposition of the input tensor, but just needs a guess on the nuclear norm. Another important difference is that the algorithm of this section is the only algorithm in this paper that actually finds the ``latent'' rank-one tensors, and hence is computationally more expensive. Our main purpose is to obtain an alternative proof of \cref{thm:intro2} for the case Euclidean norms with an easier argument, and we don't hope for computational efficiency.  On the other hand, popular tensor decomposition methods, such as \cite{kolpa}, report practical efficiency and at the same time involve similar expensive optimization subroutines as the one used in Algorithm \ref{alg:FW}. This suggests there might be room for experimentation to see if Algorithm \ref{alg:FW} is useful for particular benchmark problems, which we have to leave to the interested readers due to time constraints.

The algorithm is based on optimizing an objective function on the Veronese body that was defined in Section \ref{nuclear-norm}. More precisely, given $q \in V_{n,d}$ we consider the objective function    
    \begin{equation*}    \label{objective-FW}
    F(p):= \frac{1}{2}\norm{p-q}^2_{\rm HS},
    \end{equation*}
and we minimize the objective function on $V_{n,d}$. The algorithm, in return, constructs a low-rank approximation of $q$, and the number of steps taken by the algorithm controls the rank of its output.

Each recursive step in the algorithm is solved directly over the constraint set $V_{n,d}$: So every linear function involved attains the minimum at some extreme point of $V_{n,d}$ given by $ \pm v \otimes \ldots \otimes v$ for some $v \in S^{n-1}$. Therefore, the $h_i$'s produced in step 5 are always rank-1 symmetric tensors. In the end, the number of steps of the algorithm controls the ${\rm srank}$ of the output $p_k$.
\begin{lemma} \label{fw:stepsize}
Algorithm \ref{alg:FW} terminates in at most $\ceil{8/\varepsilon^2}$  many steps.
\end{lemma}
\begin{proof}[Proof of Lemma \ref{fw:stepsize}]
Recall that $F(p)=\frac{1}{2}\norm{p-q}_{\rm HS}^2$, so we have that $\nabla F (p)=  -q + p$ for all $p$. Therefore, for every $g_1$ and $g_2$ we have
\[ F(g_2) - F(g_1) = \frac{1}{2} \norm{g_2-g_1}_{\rm HS}^2 + \langle g_2 - g_1 , \nabla F(g_1)  \rangle_{\rm HS}. \]
This gives the following:
\begin{align*}
        F(p_{k+1})-F(p_k) 
        &= \ang{p_{k+1}-p_k, \nabla F(p_k)} + \frac{1}{2}\norm{p_{k+1}-p_k}^2_{\rm HS}\\
        &= \gamma_k\ang{h_k-p_k,\nabla F(p_k)}  + \frac{1}{2}\gamma_k^2\norm{h_k-p_k}^2_{\rm HS}\\
        &\leq \gamma_k\ang{h_k-p_k,\nabla F(p_k)}  + 2  \gamma_k^2 \\
        &\leq \gamma_k\ang{q-p_k,\nabla F(p_k)} + 2 \gamma_k^2\\
        &\leq \gamma_k(F(q)-F(p_k)) + 2  \gamma_k^2 .
\end{align*}
    Setting $\delta_k = F(p_k)-F(q)$ the inequality reads
\begin{equation*}
\delta_{k+1} - \delta_k  \leq - \gamma_k \delta_k  + 2  \gamma_k^2
\end{equation*}
that is
\begin{equation*}
\delta_{k+1} \leq (1-\gamma_k)\delta_k + 2 \gamma_k^2 .
\end{equation*}
Using $\gamma_k=\frac{2}{k+1}$  we obtain
    \[ F(p_{k+1})-F(q) \leq \frac{8}{k+1}.
    \]
Hence, given a desired level of accuracy $\varepsilon>0$ the algorithm terminates in at most $\ceil{\frac{8}{\varepsilon^2}}$ steps.
\end{proof}
Note that for any $f \in P_{n,d}$ we have $\frac{f}{\norm{f}_{*}} \in V_{n,d}$. Thus as a corollary of Lemma \ref{fw:stepsize}, using $\frac{\varepsilon}{\norm{f}_{*}}$, we obtain the following rank estimate.
\begin{corollary} \label{HSrank}
Let $f \in P_{n,d}$, then we have
\[ {\rm srank}_{\norm{\cdot}_{\rm HS}, \varepsilon} (f) \leq \frac{8 \norm{f}_{*}^2}{\varepsilon^2}. \]
\end{corollary}

Lemma~\ref{fw:stepsize} controls the number of steps in the Frank-Wolfe type algorithm. Thus, the remaining piece in complexity analysis is to understand the computational complexity of Step 6. First, we observe that $\nabla F(p_k)=-q+p_k$ and $h_k = \underset{h \in V_{n,d}}{\argmin} \langle h , p_k-q \rangle_{\rm HS}$. In other words, $h_k= q_{v_k}$ for which we have $(q-p_k)(v_k) = \max_{v \in S^{n-1}} (q-p_k)(v)$. Therefore, finding $h_k$ is equivalent to optimizing $q-p_k$ on the sphere $S^{n-1}$. This optimization step is indeed expensive (NP-Hard for $d \geq 4$). Here we content ourselves by providing an estimate on the complexity of Step 6. 
\begin{lemma} \label{covering}
Given $p \in P_{n,d}$, one can find $v \in S^{n-1}$ with
\[ 
    |p(v)| \leq \max_{z \in S^{n-1}} |p(z)| \leq  \frac{1}{1-\eta^2} |p(v)|
\]
by computing at most $O(( 3 d/ \eta)^n)$ many pointwise evaluations of $p$ on $S^{n-1}$.
\end{lemma}
This lemma follows from a standard covering argument, see Proposition 4.5 of \cite{cucker2021} for an exposition. An alternative approach to polynomial optimization is the sum of squares (SOS) hierarchy: For the case of optimizing a polynomial on the sphere using SOS,  the best current result seems to be \cite[Theorem 1]{fawzi}. This result shows that SOS produces a constant error approximation to $\norm{p}_{\infty}$ of a degree-$d$ symmetric tensor $p$ with $n$ variables in its $(nc_n)$-th layer, where $c_n$ is a constant depending on $n$ . In terms of algorithmic complexity, this means SOS is proved to produce a constant error approximation with $O(n^n)$ complexity. So, for the cases $d < n$ the simple lemma above seems stronger than state of the art theorems for the sum of squares approach.
\begin{remark}
The Frank-Wolfe algorithm in this section is quite natural. However, we couldn't locate any earlier use of this algorithm for symmetric tensor decomposition. We do not know the earliest appearance of this idea in different settings; as far as we are able to locate, the beautiful paper \cite{pokutta} deserves the credit.
\end{remark}

\section{An application to optimization}
This section concerns the optimization of symmetric $d$-tensors for even $d$ when $\norm{p}_{*}$ is small. Suppose one has $p=\sum_{i} c_i v_i \otimes v_i \otimes \ldots \otimes v_i$ where $\sum_{i} \abs{c_i} \leq c \norm{p}_{*}$ for some constant $c$. If we are given a decomposition with this property then we can approximate $\norm{p}_{\infty}$ in a reasonably fast and accurate way: We first apply Algorithm \ref{alg:empirical-approx} to $p$, that is, we compute $q \in P_{n,d}$ such that $\norm{p-q}_{HS} \leq \varepsilon$ and
    \[
    q = \frac{1}{m} \sum_{i=1}^m v_i \otimes v_i \otimes \ldots \otimes v_i,
    \]
where ${\rm srank}(q) = m  \leq \ceil{\frac{c^2 \norm{p}_{*}^2}{\varepsilon^2}}$. Also notice that
    \[ 
    \abs{ \norm{p}_{\infty} - \norm{q}_{\infty} }  \leq  \norm{p-q}_{\infty} \leq \norm{p-q}_{\rm HS} \leq  \varepsilon.
    \]
The next step is to compute $\norm{q}_{\infty}$ and an approach is offered by Lemma~\ref{lem:Barv-ineq}. First, observe that 
    \[ \norm{q}_{2k}^{2k} =  \frac{1}{m^{2k}} \sum_{1 \leq i_1, i_2, \ldots, i_k \leq m} \int_{S^{n-1}} \prod_{j=1}^k \langle x , v_{i_j} \rangle^d \;\sigma(x) \]
and note that there are $\binom{m+k-1}{k} = O(k^m)$ many summands in the expression of $\norm{q}_{2k}^{2k}$. In addition, the values of these  summands are given by a Gamma-like function at the vectors $v_1,v_2,\ldots, v_m$. Second, observe that  for $k \gtrsim \frac{n}{\varepsilon}\ln(ed/\varepsilon)$ we have $(edk/n)^{\frac{n}{2k}}<1+\varepsilon$. So, for $k > \frac{cn}{\varepsilon} \ln(\frac{ed}{\varepsilon})$ using Lemma \ref{lem:Barv-ineq} and Stirling's estimate one has
\[ 
    \norm{q}_{2k} \leq \norm{q}_{\infty} \leq \left( \frac{edk}{n} \right)^{\frac{n}{2k}} \norm{q}_{2k} \leq  (1+\varepsilon) \norm{q}_{2k}.  
\]
In return, for $k\asymp \frac{n}{\varepsilon} \ln(\frac{ed}{\varepsilon})$ we can calculate
\[  
    \norm{q}_{2k} - \varepsilon \leq \norm{p}_{\infty}  \leq (1+\varepsilon) \norm{q}_{2k} + \varepsilon  
\]
by computing $O\left( (\frac{n \ln (ed)}{\varepsilon^2})^m \right)$ many summands. In principle, this approach gives an algorithm that operates in time $O \left( (n \ln(ed))^{\frac{c^2 \norm{p}_{*}^2}{\varepsilon^2}} \right)$.  However, one must be aware of potential numerical issues due to integration of high degree terms.

In addition to the above, there is an alternative approach coming from \cite{alp} with advantages in numerical computations. After we compute $ q = \frac{1}{m} \sum_{i=1}^m  v_i \otimes v_i \otimes \ldots \otimes v_i$, it is possible to exploit the fact that $q \in W:= {\rm span} \{ v_i \otimes v_i \otimes \ldots \otimes v_i : 1 \leq i \leq m \}$ and  $\dim W \leq \ceil{\frac{c^2 \norm{p}_{*}^2}{\varepsilon^2}}$. The approach presented in Theorem 1.6 of \cite{alp} gives a $1-\frac{1}{n}$ approximation to $\norm{q}_{\infty}$ using $O \left( n^{\frac{c^2 \norm{p}_{*}^2 }{\varepsilon^2}} \right)$ many pointwise evaluations. Moreover, this approach has the advantage of being simple and using only degree-$d$ tensors.  The following theorem summarizes the discussion in this section.
\begin{theorem}
Let $p=\sum_{i} c_i v_i \otimes v_i \otimes \ldots \otimes v_i$ where $\sum_{i} \abs{c_i} \leq c \norm{p}_{*}$ . Then using  Algorithm \ref{alg:empirical-approx} and  the results of \cite{alp}:
\begin{itemize}
     \item we compute a $q \in P_{n,d}$ such that  ${\rm srank} (q) \leq \frac{c^{2}\norm{p}_{*}^2}{\varepsilon^2}$ and $\abs{\norm{p}_{\infty}-\norm{q}_{\infty}} \leq \varepsilon$, 
    \item we compute a $1-\frac{1}{n}$ approximation of $\norm{q}_{\infty}$, with high probability, using $O( n^{\frac{c^2 \norm{p}_{*}^2}{\varepsilon^2}} )$ many pointwise evaluations of $q$ on the sphere $S^{n-1}$.
\end{itemize}  
\end{theorem}

\section{Acknowledgments}
We thank Jiawang Nie for answering our questions on optimization of low-rank symmetric tensors using sum of squares. We thank Sergio Cristancho and Mauricio Velasco for explaining the mathematical underpinning of their quadrature rule in \cite{mauricio}, and allowing us to implement it in Python. We thank Carlos Castro Rey for his valuable help and feedback on the Python implementation. A.E. was partially supported by NSF CCF 2110075. P.V. is supported by Simons Foundation grant 638224.

\bibliography{approxtensor.bib}
\bibliographystyle{amsalpha}

\end{document}